\documentclass{amsart}

\usepackage{ragged2e}
\usepackage{a4wide}
\usepackage[utf8]{inputenc}
\usepackage{enumerate}

\usepackage{amsmath}
\usepackage{amsfonts}
\usepackage{mathabx}

\usepackage{bbm}
\usepackage{mathrsfs}
\usepackage{amsthm}
\usepackage{verbatim}
\usepackage{upgreek}
\usepackage{relsize}
\usepackage{dsfont}
\usepackage{graphicx}
\usepackage{amssymb}
\usepackage{xcolor}

\newcommand{\eps}{\varepsilon}
\newcommand{\ov}{\overline}
\newcommand{\id}{\textnormal{id}}
\newcommand{\mc}{\mathcal}

\newcommand{\mrm}{\mathrm}

\newcommand{\msf}{\mathsf}
\newcommand{\I}{\mathbbm{1}}

\newcommand{\vp}{\varphi}

\newcommand{\cst}{\ifmmode \mathrm{C}^* \else $\mathrm{C}^*$\fi}

\newcommand{\la}{\langle}
\newcommand{\ra}{\rangle}

\newcommand{\bbGamma}{{\mathpalette\makebbGamma\relax}}
\newcommand{\makebbGamma}[2]{%
  \raisebox{\depth}{\scalebox{1}[-1]{$\mathsurround=0pt#1\mathbb{L}$}}%
}

\newcommand{\NN}{\mathbb{N}}

\newcommand{\CC}{\mathbb{C}}

\newcommand{\GG}{\mathbb{G}}

\newcommand{\wot}{\ifmmode \textsc{wot} \else \textsc{wot}\fi}
\newcommand{\sot}{\ifmmode \textsc{sot} \else \textsc{sot}\fi}
\newcommand{\sots}{\ifmmode \textsc{sot}^* \else \textsc{sot}$^*$\fi}
\newcommand{\ssot}{\ifmmode \sigma\textsc{-sot} \else $\sigma$-\textsc{sot }\fi}
\newcommand{\ssots}{\ifmmode \sigma\textsc{-sot}^* \else $\sigma$-\textsc{sot }$^*$\fi}
\newcommand{\swot}{\ifmmode \sigma\textsc{-wot} \else $\sigma$-\textsc{wot}\fi}

\newcommand{\Linfd}{\operatorname{L}^{\infty}(\wh\bbGamma)}
\newcommand{\Ljd}{\operatorname{L}^{1}(\wh\bbGamma)}
\newcommand{\CGd}{\mathrm{C}(\wh\bbGamma)}

\newcommand{\wh}{\widehat}
\newcommand{\wc}{\widecheck}

\newcommand{\oon}{\operatorname}

\renewcommand{\restriction}{\mathord{\upharpoonright}}
\newcommand{\rest}{\restriction}

\DeclareMathOperator{\lin}{span}
\DeclareMathOperator{\Irr}{Irr}
\DeclareMathOperator{\Pol}{Pol}

\DeclareMathOperator{\Tr}{Tr}
\DeclareMathOperator{\B}{B}
\DeclareMathOperator{\A}{A}

\DeclareMathOperator{\M}{M}

\DeclareMathOperator{\LL}{L}

\DeclareMathOperator{\SU}{SU}
\DeclareMathOperator{\CB}{CB}

\newtheorem{theorem}{Theorem}[section]

\newtheorem{proposition}[theorem]{Proposition}
\newtheorem{lemma}[theorem]{Lemma}
\theoremstyle{definition}
\newtheorem{corollary}[theorem]{Corollary}
\newtheorem{remark}[theorem]{Remark}

\newtheorem{definition}[theorem]{Definition}
\newtheorem{example}[theorem]{Example}
\numberwithin{equation}{section}

\begin{document}
\title{Cowling-Haagerup constant of the product of discrete quantum groups}

\author{Jacek Krajczok}
\address{Vrije Universiteit Brussel\\
Pleinlaan 2\\
1050 Brussels\\
Belgium
}
\email{jacek.krajczok@vub.be}

\thanks{}




\date{}

\begin{abstract}
We show that (central) Cowling-Haagerup constant of discrete quantum groups is multiplicative $\Lambda_{cb}(\bbGamma_1\times \bbGamma_2)=\Lambda_{cb}(\bbGamma_1)\,\Lambda_{cb}(\bbGamma_2)$, which extends the result of Freslon \cite{FreslonPermanence} to general (not necesarilly unimodular) discrete quantum groups. The crucial feature of our approach is considering algebras $\CGd,\Linfd$ as operator modules over $\Ljd$.
\end{abstract}

\maketitle

\section{Introduction}

Weak amenability is an approximation property introduced in the context of locally compact groups by Cowling and Haagerup in \cite{CowlingHaagerup}. It is weaker then amenability, but still quite strong as it implies Haagerup-Kraus approximation property (AP). A significant aspect of weak amenability is that it comes together with a quantifier: for any locally compact group one defines Cowling-Haagerup constant $\Lambda_{cb}(G)\in \left[1,+\infty\right]$ which is finite precisely when $G$ is weakly amenable. Authors of \cite{CowlingHaagerup, Haagerup} calculated this constant for all connected, non-compact, simple Lie groups with finite center. For example $\Lambda_{cb}(\oon{Sp}(1,n))=2n-1\,(n\ge 2)$ but if real rank of $G$ is greater than one, then $G$ is not weakly amenable and $\Lambda_{cb}(G)=+\infty$. Another important result tells that if $\Gamma$ is a lattice in $G$ then $\Lambda_{cb}(\Gamma)=\Lambda_{cb}(G)$, hence Cowling-Haagerup constant is a useful tool in telling apart discrete groups and their group \cst/von Neumann algebras. Cowling and Haagerup proved also that constant $\Lambda_{cb}$ is multiplicative, i.e. $\Lambda_{cb}(G\times H)=\Lambda_{cb}(G)\Lambda_{cb}(H)$ holds for any locally compact groups $G,H$ (\cite[Corollary 1.5]{CowlingHaagerup}).

One can extend the definition of weak amenability and Cowling-Haagerup constant to discrete or even general locally compact quantum groups (\cite[Definition 3.5]{Freslon}, \cite[Definition 5.12]{Brannan}, see also Definition \ref{def1}). This property has received a lot of attention -- let us mention that it is known that strong amenability (i.e.~coamenability of the dual) implies weak amenability which in turn implies AP, weak amenability with $\Lambda_{cb}=1$ is preserved under taking free products of discrete quantum groups \cite{FreslonFreeProducts} and quantum groups such as $\wh{O_F^+}, \wh{U_F^+}$ or $\SU_q(1,1)_{ext}$ are weakly amenable with Cowling-Haagerup constant equal $1$ (\cite[Theorem 24]{CCAP}, \cite[Theorem 7.4]{Caspers}). It is however an open question whether amenability implies weak amenability (in fact it is not known whether amenability implies AP, see \cite[Corollary 7.4]{CrannInner}. These implications are known to be true in discrete case by \cite[Theorem 3.8]{Tomatsu}). Freslon in \cite[Proposition 3.2]{FreslonPermanence} proved that weak amenability passes to products of discrete quantum groups, but so far the best information on the value of Cowling-Haagerup constant were the bounds $\max(\Lambda_{cb}(\bbGamma_1),\Lambda_{cb}(\bbGamma_2))\le \Lambda_{cb}(\bbGamma_1\times \bbGamma_2)\le \Lambda_{cb}(\bbGamma_1)\Lambda_{cb}(\bbGamma_2)$. In Theorem \ref{thm3} we will show that the upper bound $\le $ is in fact always an equality. Example \ref{ex1} shows why this knowledge can make a qualitative difference.

For discrete quantum groups there is a close connection between properties of quantum group $\bbGamma$ and its operator algebras $\CGd,\Linfd$. For example, weak amenability of $\bbGamma$ implies that $\CGd$ has completely bounded approximation property, $\Linfd$ has weak$^*$ completely bounded approximation property and there is a bound on respective constants (see \cite[Theorem 6.6]{Brannan} and references therein). The converse holds under unimodularity assumption (\cite[Theorem 5.14]{KrausRuan}) and in this case all the involved constants are equal (see also \cite[Proposition 4.7]{CrannInner} for a related result). Whether this converse and its variants for strong amenability and AP hold in general, is a major open problem (\cite[Remark 6.9]{Brannan}). The main reason why in general it is difficult to deduce a property of $\bbGamma$ from properties of $\CGd,\Linfd$ is the lack of averaging which exists in unimodular (dually - Kac type) case, and allows one to turn a CB map into a multiplier (see \cite[Section 7.1]{Brannan} and \cite[Section 7.1]{Averaging}). As Freslon notes in \cite[Remark 3.3]{FreslonPermanence}, in the unimodular case we can use equality $\Lambda_{cb}(\bbGamma)=\Lambda_{cb}(\CGd)$ to deduce that Cowling-Haagerup constant is multiplicative using \cite[Theorem 12.3.13]{BrownOzawa}. This result states that Cowling-Haagerup constant of \cst-algebras is multiplicative with respect to minimal tensor product. In general however this approach does not work, as it is not known whether $\Lambda_{cb}(\bbGamma)\le \Lambda_{cb}(\mrm{C}(\wh\bbGamma))$.
One way of remeding this situation is to look at $\CGd, \Linfd$ not only as at \cst/von Neumann algebras, but consider them together with extra structure. This approach already turned out to be quite fruitful and lead to several results concerning amenability -- injectivity (see \cite[Theorem 3]{SoltanViselter} and \cite[Theorem 5.1]{Crann}), AP -- weak$^*$ OAP (\cite[Theorem 6.16]{DKV_ApproxLCQG}) or strong amenability -- weak$^*$ CPAP (\cite[Theorem 6.11]{KrajczokPhD}).

In our work we take a similar point of view, and look at $\CGd,\Linfd$ as $\Ljd$-modules. In Definition \ref{def1} we introduce respective Cowling-Haagerup-like constants and in Theorem \ref{thm3} show that they are equal to the analogous constants for $\bbGamma$. In Section \ref{sec:main} we show that such Cowling-Haagerup constant for operator modules of the form $\CGd$ is multiplicative (Proposition \ref{prop2}). Its proof is a modification of the proof of \cite[Theorem 12.3.13]{BrownOzawa}. The main difference is that we take also the module structure into account (see also remarks \ref{rem1}, \ref{rem2}).

Apart from weak amenability of discrete quantum groups, we are also interested in its central variation (see Definition \ref{def1}). To study this property, we will look at $\CGd,\Linfd$ as $\Ljd$-bimodules.

\section{Preliminaries and notation}\label{sec:preliminaries}

In this section we will briefly recall the necessary operator space and quantum group background. We refer to \cite{BlecherMerdy, Crann, EffrosRuan}, \cite{bmt, Brannan, DKV_ApproxLCQG, KustermansVaesVN, NeshveyevTuset, PodlesWoronowicz, cqg} and references therein for more information.\\

Completely contractive Banach algebra is an associative algebra $A$ which is at the same time an operator space and the multiplication map extends to a complete contraction $A\wh\otimes A\rightarrow A$, where $\wh\otimes $ is the projective tensor product of operator spaces. We say that an operator space $X$ is a left operator $A$-module, if it is a left module over $A$ and the action extends to a complete contraction $A\wh\otimes X\rightarrow X$. Since this is the only type of modules we consider, we will simply say that $X$ is a left $A$-module. In a similar way we define right $A$-modules and $A$-$B$-bimodules. By definition, an $A$-bimodule is an $A$-$A$-bimodule. Note that every operator space or module can be considered as a bimodule by setting $A=\CC$, $B=\CC$ or both. Furthermore, if $A,B$ are completely contractive Banach algebras, then so is $A\wh\otimes B$.

The operator space of completely bounded (CB) maps between two operator spaces $X,Y$ will be denoted by $\CB(X,Y)$. If $X,Y$ are left $A$-modules, then the closed subspace consisting of left $A$-module maps will be denoted by ${}_A\CB(X,Y)$. Similarly we define the space of right $A$-module maps $\CB_A(X,Y)$ and $A$-$B$-bimodule maps ${}_A\CB_B(X,Y)$. The CB norm will be denoted by $\|\vp\|_{\CB(X,Y)}$ or simply $\|\vp\|_{cb}$.

If $A$ is a completely contractive Banach algebra and $X$ is a left $A$-module, then the dual operator space $X^*$ becomes canonically a right $A$-module with action defined by $\la\omega a,x\ra=\la\omega,ax\ra$. Similarly for right modules and bimodules. The canonical pairing between $X^*$ and $X$ will be denoted simply by $\la\omega,x\ra$ or $\la \omega,x\ra_{X^*,X} $ if we want to indicate which spaces are involved. Pairing gives rise to canonical complete contraction $\kappa\colon X\wh\otimes X^*\rightarrow \CC$.

Let $X,Y$ be operator spaces, $X$ a right $A$-module, and $Y$ a left $A$-module. Then we can form the $A$-module tensor product $X\wh\otimes_A Y$, which by definition is given by the quotient operator space
\[
X \wh\otimes_A Y=(X\wh\otimes Y)/\ov{\lin} \{ xa\otimes y - x\otimes ay\,|\, x\in X, a\in A, y\in Y\}.
\] 
By an abuse of notation, the quotient map will be denoted by $q\colon X\wh\otimes Y\rightarrow X\wh\otimes_A Y$. A result which will be very useful, is that in this situation $\CB_A(X,Y^*)\simeq (X\wh\otimes_A Y)^*$ completely isometrically, where $\vp\in \CB_A(X,Y^*)$ corresponds to the functional $q(x\otimes y)\mapsto \la \vp(x),y\ra $ (\cite[Proposition 3.5.9]{BlecherMerdy}). Similarly $\CB(X,Y^*)\simeq (X\wh\otimes Y)^*$ completely isometrically. In this way both $\CB_A(X,Y^*)$ and $\CB(X,Y^*)$ are dual operator spaces and have the corresponding weak$^*$ topologies. In particular, one can restrict weak$^*$ topology from $\CB(X,Y^*)$ to $\CB_A(X,Y^*)$. One easily checks that both topologies on $\CB_A(X,Y^*)$ agree and $\CB_A(X,Y^*)$ is weak$^*$ closed in $\CB(X,Y^*)$.

If $A$ is a completely contractive Banach algebra, then so is $A^{op}$ ($A^{op}$ by definition has the same operator space structure, but opposite multiplication). Then any left $A$-module becomes right $A^{op}$-module and vice versa. Furthermore, if $X$ is a $A$-$B$-bimodule then it is a right $A^{op}\wh\otimes B$-module, with module structure $x(a^{op}\otimes b)=axb$. One immediately sees that ${}_A\CB_B(X,Y)\subseteq \CB_{A^{op}\wh\otimes B}(X,Y)$ for any $A$-$B$-bimodules $X,Y$ (see also Remark \ref{rem3}). Let us also recall that for any finite dimensional operator space $E$, the canonical map $E\rightarrow E^{**}$ establishes a completely isometric isomorphism.\\

In this work we will be interested only in compact or discrete quantum groups. Readers interested in general framework are referred to \cite{KustermansVaesVN}. Compact quantum group $\GG$ is defined by a unital \cst-algebra $\mrm{C}(\GG)$ and a unital $*$-homomorphism $\Delta\colon \mrm{C}(\GG)\rightarrow \mrm{C}(\GG)\otimes \mrm{C}(\GG)$ called comultiplication, which satisfies certain conditions. Under separability assumption Woronowicz \cite{cqg} (and van Daele \cite{VanDaeleHaar} in general) proved that there exists a unique state $h\in \mrm{C}(\GG)^*$ (called Haar integral) which is bi-invariant. We will assume that it is faithful, i.e.~we work at the reduced level (see \cite{bmt}). Performing GNS representation, we obtain Hilbert space $\LL^2(\GG)$, faithful representation of $\mrm{C}(\GG)$ and after taking {\sot}-closure, von Neumann algebra $\LL^{\infty}(\GG)$. Both $h$ and $\Delta$ extend to normal maps on $\LL^{\infty}(\GG)$. The predual of $\LL^{\infty}(\GG)$ will be denoted by $\LL^1(\GG)$. The predual mapping of $\Delta$ gives a completely contractive Banach algebra structure on $\LL^1(\GG)$:
\[
\LL^1(\GG)\wh\otimes \LL^1(\GG)\ni \omega\otimes \nu\mapsto\omega\star\nu=(\omega\otimes\nu)\Delta\in \LL^1(\GG).
\]
It is not difficult to check that both $\mrm{C}(\GG)$ and $\LL^{\infty}(\GG)$ are $\LL^1(\GG)$-bimodules with respect to actions $\omega\star x=(\id\otimes \omega)\Delta(x)$, $x\star \omega = (\omega\otimes \id)\Delta(x)$ for $\omega\in \LL^1(\GG)$ and $x\in \mrm{C}(\GG)$ or $x\in \LL^{\infty}(\GG)$. Representation theory of compact quantum groups resembles the one of compact qroups. In particular, every irreducible representation is finite dimensional. Let $\Irr(\GG)$ be the set of their equivalence classes. For each class $\alpha\in\Irr(\GG)$ we choose its representative $U^{\alpha}$ which acts on a Hilbert space $\msf{H}_\alpha$ of dimension $\dim(\alpha)$. In each $\msf{H}_\alpha$ choose an orthonormal basis $\{\xi^\alpha_i\}_{i=1}^{\dim(\alpha)}$ in which operator $\uprho_\alpha$ is diagonal (see \cite[Section 1.4]{NeshveyevTuset}), with eigenvalues $\uprho_{\alpha,i}\,(1\le i\le \dim(\alpha))$. Number $\Tr(\uprho_\alpha)$ is called the quantum dimension of $\alpha$ and is denoted $\dim_q(\alpha)$. The space $\Pol(\GG)$ spanned by coefficients $U^{\alpha}_{i,j}=(\id\otimes \omega_{\xi^\alpha_i,\xi^\alpha_j})U^{\alpha}\,(1\le i,j\le\dim(\alpha))$, together with restricted comultiplication, is a unital Hopf $*$-algebra. It is norm dense in $\mrm{C}(\GG)$, hence weak$^*$ dense in $\LL^{\infty}(\GG)$. 

By definition, any discrete quantum group $\bbGamma$ is a dual of compact quantum group $\GG$: $\bbGamma=\wh{\GG}$ (thus also $\GG=\wh\bbGamma$ -- we will prefer to look from discrete point of view). It comes together with \cst-algebra $\mrm{c}_0(\bbGamma)=\bigoplus_{\alpha\in\Irr(\wh\bbGamma)}\B(\msf{H}_\alpha)$ ($\mrm{c}_0$-direct sum), von Neumann algebra $\ell^{\infty}(\bbGamma)=\prod_{\alpha\in\Irr(\wh\bbGamma)}\B(\msf{H}_\alpha)$ and comultiplication $\Delta$. Consequently any element of $\ell^{\infty}(\bbGamma)$ is given by a family $(a_\alpha)_{\alpha\in\Irr(\wh\bbGamma)}$ of matrices in $\B(\msf{H}_\alpha)$. We will say that a net $(a_\lambda)_{\lambda\in \Lambda}$ converges pointwise to some $a$ in $\ell^{\infty}(\bbGamma)$ if and only if $a_{\lambda,\alpha}\xrightarrow[\lambda\in\Lambda]{}a_\alpha$ in $\B(\msf{H}_\alpha)$ for all $\alpha\in\Irr(\wh\bbGamma)$. The dense subspace consisting of families $(a_\alpha)_{\alpha\in\Irr(\wh\bbGamma)}$ such that $a_\alpha\neq 0$ for only finitely many $\alpha$'s, will be denoted by $\mrm{c}_{00}(\bbGamma)$. Another important subspace of $\ell^{\infty}(\bbGamma)$ is $\A(\bbGamma)$, the Fourier algebra of $\bbGamma$. It consists of elements of the form $\wh\lambda(\omega)$ with $\omega\in \LL^1(\wh\bbGamma)$ (see \cite[Section 4.2]{Brannan}, \cite[Section 3]{DKV_ApproxLCQG}). It is a subalgebra of $\mrm{c}_{0}(\bbGamma)$ and is itself a completely contractive Banach algebra with operator space structure given by completely isometric isomorphism $\A(\bbGamma)\ni \wh\lambda(\omega)\mapsto \omega\in\LL^1(\wh\bbGamma)$. A (left) completely bounded multiplier is an element $a\in \ell^{\infty}(\bbGamma)$ such that $a b\in \A(\bbGamma)$ for all $b\in \A(\bbGamma)$ and the associated map $\A(\bbGamma)\rightarrow \A(\bbGamma)$ is completely bounded. After composing with isomorphism $\A(\bbGamma)\simeq \LL^1(\wh\bbGamma)$ and taking the dual map, any such $a$ gives a normal CB map $\Theta^l(a)\in \CB^\sigma(\Linfd)$ (superscript $\sigma$ indicates that $\CB^\sigma(\Linfd)$ consists of normal CB maps). The space of completely bounded multipliers, equipped with the CB norm $\|a\|_{cb}=\|\Theta^l(a)\|_{cb}$, is denoted by $\M^l_{cb}(\A(\bbGamma))$. For example, any $\wh\lambda(\omega)\in \A(\bbGamma)$ is a left completely bounded multiplier with the associated map $\Theta^l(\wh\lambda(\omega))=(\omega\otimes\id)\wh\Delta$. Let us also note $\mrm{c}_{00}(\bbGamma)\subseteq \A(\bbGamma)$. For any $a\in \M^l_{cb}(\A(\bbGamma))$, we have $\Theta^l(a)\in {}_{\Ljd}\CB^\sigma(\Linfd)$, i.e.~$\Theta^l(a)$ is a normal, CB, left $\Ljd$-module map. By \cite[Corollary 4.4]{JungeNeufangRuan} (see also discussion in \cite[Section 3]{DKV_ApproxLCQG}) all maps on $\Linfd$ which satisfy these properties are of the form $\Theta^l(a)$ for some $a\in \M^l_{cb}(\A(\bbGamma))$. It is not difficult to check that $\Theta^l(a)$ restricts to $\Theta^l(a)\rest_{\mrm{C}(\wh\bbGamma)}\in {}_{\Ljd}\CB(\CGd)$. Using e.g.~\cite[Proposition 3.5]{DKV_ApproxLCQG} we again see that every CB, left $\Ljd$-module map on $\CGd$ is of the form $\Theta^l(a)\rest_{\mrm{C}(\wh\bbGamma)}$ for some $a\in \M^l_{cb}(\A(\bbGamma))$. Similarly, central multipliers $a\in \mc{Z}\M^l_{cb}(\A(\bbGamma))$ correspond to CB, $\Ljd$-bimodule maps on $\CGd$ and normal, CB, $\Ljd$-bimodule maps on $\Linfd$.

Whenever we have two compact quantum groups $\wh\bbGamma_1,\wh\bbGamma_2$, we can form their product $\wh\bbGamma=\wh\bbGamma_1\times \wh\bbGamma_2$. The associated algebras are $\CGd=\mrm{C}(\wh\bbGamma_1)\otimes\mrm{C}(\wh\bbGamma_2)$, $\LL^{\infty}(\wh\bbGamma)=\LL^{\infty}(\wh\bbGamma_1)\bar\otimes\LL^{\infty}(\wh\bbGamma_2)$ (hence $\LL^1(\wh\bbGamma)=\LL^1(\wh\bbGamma_1)\wh\otimes \LL^1(\wh\bbGamma_2)$), $\Pol(\wh\bbGamma)=\Pol(\wh\bbGamma_1)\odot\Pol(\wh\bbGamma_2)$ and the Haar integral is $h_{\wh\bbGamma}=h_{\wh\bbGamma_1}\otimes h_{\wh\bbGamma_2}$. We can also identify irreducible representations of $\wh\bbGamma$:
$\Irr(\wh\bbGamma)$ is the set of $\alpha\boxtimes \beta$ for $\alpha\in\Irr(\wh\bbGamma_1),\beta\in \Irr(\wh\bbGamma_2)$, where $U^{\alpha\boxtimes \beta}=U^{\alpha}_{13} U^{\beta}_{24}$ is a representation of $\wh\bbGamma$ on $\msf{H}_\alpha\otimes \msf{H}_\beta$. For details see \cite{Wang}. For finite subsets $F_1\subseteq\Irr(\wh\bbGamma_1),F_2\subseteq\Irr(\wh\bbGamma_2)$ denote $F_1\boxtimes F_2=\{\alpha\boxtimes\beta\,|\,\alpha\in F_1,\beta\in F_2\}$. Product of discrete quantum groups $\bbGamma_1$ and $\bbGamma_2$ is defined to be $\bbGamma_1\times \bbGamma_2=\bbGamma$, where $\bbGamma$ is the dual of $\wh\bbGamma$.

We will be using the following useful notation: if $\wh\bbGamma$ is an arbitrary compact quantum group and $\emptyset\neq F\subseteq \Irr(\wh\bbGamma)$ is a finite subset, set $\Pol_F(\wh\bbGamma)=\lin\{U^{\alpha}_{i,j}\,|\, \alpha\in F,1\le i,j\le \dim(\alpha)\}$ and consider it to be an operator space with structure coming from $\CGd$.  Next, for each $\alpha\in \Irr(\wh\bbGamma)$, let $p_\alpha$ be the central projection corresponding to $\B(\msf{H}_\alpha)\subseteq \ell^{\infty}(\bbGamma)$ and $p_F=\sum_{\alpha\in F}p_\alpha\in \mrm{c}_{00}(\bbGamma)$. Using orthogonality relations one easily sees that \begin{equation}\label{eq10}
 p_F=\wh\lambda(\omega_F),\quad\textnormal{ where }\quad
 \omega_F=\sum_{\alpha\in F} \sum_{i=1}^{\dim(\alpha)} \dim_q(\alpha) \uprho_{\alpha,i} h(U^{\alpha *}_{i,i}\cdot)\in \LL^1(\wh\bbGamma).
 \end{equation}
 Furthermore, $\Theta^l(p_F)$ is a projection onto $\Pol_F(\wh\bbGamma)$.\\

Symbol $\odot$ will denote the algebraic tensor product, $\otimes $ tensor product of Hilbert spaces or minimal (spatial) tensor product of \cst-algebras, $\bar\otimes$ von Neumann algebraic tensor product and $\wh\otimes$ projective tensor product of operator spaces. Operator spaces are assumed to be complete. All vector spaces are considered over $\CC$.

\section{Cowling-Haagerup constant for modules}
In this section we introduce a Cowling-Haagerup constant for (bi)modules $\CGd,\Linfd$, study its properties and relate it to the (central) Cowling-Haagerup constant of $\bbGamma$ (Theorem \ref{thm3}).

\begin{definition}\label{def2}
Let $\bbGamma$ be a discrete quantum group.
\begin{itemize}
\item Define ${}_{\Ljd}\Lambda_{cb}(\CGd)$ to be the infimum of all numbers $C\ge 1$ such that there is a net $(\vp_\lambda)_{\lambda\in \Lambda}$ of finite rank, left $\Ljd$-module CB maps on $\CGd$ with $\|\vp_{\lambda}\|_{cb}\le C$ and $\vp_{\lambda}(x)\xrightarrow[\lambda\in \Lambda]{} x$ for all $x\in \CGd$. If no such number exists, set ${}_{\Ljd} \Lambda_{cb}(\CGd)=+\infty$.
\item Similarly define $\Lambda_{cb,\Ljd}(\CGd)$ and ${}_{\Ljd}\Lambda_{cb,\Ljd}(\CGd)$ by considering right $\Ljd$-module maps and $\Ljd$-bimodule maps, respectively.
\item Define ${}_{\Ljd}\Lambda_{cb}(\Linfd)$ to be the infimum of all numbers $C\ge 1$ such that there is a net $(\psi_{\lambda})_{\lambda\in \Lambda}$ of normal, finite rank, left $\Ljd$-module CB maps on $\Linfd$ with $\|\psi_{\lambda}\|_{cb}\le C$ and $\psi_{\lambda}(x)\xrightarrow[\lambda\in \Lambda]{} x$ weak$^*$ for all $x\in \Linfd$. If no such number exists, set ${}_{\Ljd} \Lambda_{cb}(\Linfd)=+\infty$.
\item Similarly define $\Lambda_{cb,\Ljd}(\Linfd)$ and ${}_{\Ljd}\Lambda_{cb,\Ljd}(\Linfd)$ by considering right $\Ljd$-module maps and $\Ljd$-bimodule maps, respectively.
\end{itemize}
\end{definition}

Numbers ${}_{\Ljd}\Lambda_{cb}(\CGd)$, etc.~will be called Cowling-Haagerup constants. A standard argument (using a new net indexed over $\Lambda\times \NN$) shows that the infimum in the above definition is actually achievable. 

\begin{remark}\label{rem1}
In principle we could have introduced similar constants for arbitrary operator modules over completely contractive Banach algebras. We have decided not to do that, as general operator modules can fail to have any finite dimensional submodules, and also we were unable to prove that such constant is in general multiplicative (see Proposition \ref{prop2} and Remark \ref{rem2}).
\end{remark}

In our first proposition we show that it doesn't matter if we look at left or right module structure, and similarly it doesn't matter if we look at {\cst} or von Neumann level.

\begin{proposition}\label{prop1}
Let $\bbGamma$ be a discrete quantum group. Then
\[
{}_{\LL^1(\wh\bbGamma)}\Lambda_{cb}(\CGd)=\Lambda_{cb,\LL^1(\wh\bbGamma)}(\CGd)={}_{\LL^1(\wh\bbGamma)}\Lambda_{cb}(\Linfd)=\Lambda_{cb,\LL^1(\wh\bbGamma)}(\Linfd)
\]
and
\[
{}_{\LL^1(\wh\bbGamma)}\Lambda_{cb,\Ljd}(\CGd)={}_{\LL^1(\wh\bbGamma)}\Lambda_{cb,\Ljd}(\Linfd).
\]
\end{proposition}

\begin{proof}
If $\psi\in \CB^\sigma(\Linfd)$ is a normal, left $\Ljd$-module map, then $\wh{R}\circ\psi\circ \wh{R}$ is a normal right $\Ljd$-module map with $\|\wh{R}\circ\psi\circ\wh{R}\|_{cb}=\|\psi\|_{cb}$ (\cite[Lemma 4.8]{DKV_ApproxLCQG}), where $\wh{R}$ is the unitary antipode on $\LL^{\infty}(\wh\bbGamma)$. We can similarly turn right $\Ljd$-module maps into left one, the property of being finite rank is preserved. Finally, net $(\psi_{\lambda})_{\lambda\in \Lambda}$ converges to $\id$ in the point-weak$^*$ topology if and only if $(\wh{R}\circ \psi_{\lambda} \circ \wh{R})_{\lambda\in \Lambda}$ converges to $\id$. This shows ${}_{\LL^1(\wh\bbGamma)}\Lambda_{cb}(\Linfd)=\Lambda_{cb,\LL^1(\wh\bbGamma)}(\Linfd)$.

The above quoted part of \cite[Lemma 4.8]{DKV_ApproxLCQG} has a \cst-algebraic variant (with virtually the same proof, using usual Wittstock's theorem \cite[Theorem B7]{BrownOzawa}): if $\vp\in \CB(\CGd)$ is a finite rank, left $\Ljd$-module map, then $\wh{R}\circ \vp \circ \wh{R}$ is a finite rank, right $\Ljd$-module map with the same CB norm, and vice versa. Similarly, $(\vp_\lambda)_{\lambda\in \Lambda}$ converges in point-norm topology to $\id$ if and only if $(\wh{R}\circ\vp_\lambda \circ \wh{R})_{\lambda\in \Lambda}$ does, hence ${}_{\LL^1(\wh\bbGamma)}\Lambda_{cb}(\CGd)=\Lambda_{cb,\LL^1(\wh\bbGamma)}(\CGd)$.

Assume that $(\psi_\lambda)_{\lambda\in \Lambda}$ is a net of finite rank maps in ${}_{\Ljd}\CB^\sigma(\Linfd)$ which converges to $\id$ in the point-weak$^*$ topology and $\|\psi_\lambda\|_{cb}\le {}_{\Ljd}\Lambda_{cb}(\Linfd)<+\infty$. As discussed in Section \ref{sec:preliminaries}, for each $\lambda\in \Lambda$ there is $a_\lambda\in \mrm{c}_{00}(\bbGamma)$ such that $\psi_\lambda=\Theta^l(a_\lambda)$. Then $\Theta^l(a_\lambda)$ restricts to a map $\vp_\lambda$ in ${}_{\Ljd}\CB(\CGd)$ with $\|\vp_\lambda\|_{cb}=\|\psi_\lambda\|_{cb}$ (by weak$^*$ density of $\CGd\subseteq\Linfd$) such that $\vp_\lambda(x)\xrightarrow[\lambda\in\Lambda]{} x$ in norm for every $x\in \Pol(\wh\bbGamma)$. Indeed, observe that $\{\vp_\lambda(x),x\,|\,\lambda\in \Lambda\}$ live in a finite dimensional subspace of $\Pol(\wh\bbGamma)$, and in finite dimensional spaces there is a unique Hausdorff vector space topology. Since $\Pol(\wh\bbGamma)$ is norm dense in $\CGd$ and the net $(\vp_\lambda)_{\lambda\in\Lambda}$ is bounded, a standard approximation argument allows us to conclude that $\vp_\lambda(x)\xrightarrow[\lambda\in\Lambda]{}x$ for all $x\in \CGd$ and consequently ${}_{\Ljd}\Lambda_{cb}(\CGd)\le {}_{\Ljd}\Lambda_{cb}(\Linfd)$.

Similar reasoning gives ${}_{\Ljd}\Lambda_{cb,\Ljd}(\CGd)\le {}_{\Ljd}\Lambda_{cb,\Ljd}(\Linfd)$. The only difference is that if $\psi_\lambda$ is known to be a $\Ljd$-bimodule map, then so will be $\vp_\lambda$.

Assume that $(\vp_\lambda)_{\lambda\in\Lambda}$ is a net of finite rank maps in ${}_{\Ljd}\CB(\CGd)$ with CB norm bounded by $\|\vp_\lambda\|_{cb}\le {}_{\Ljd} \Lambda_{cb}(\CGd)$, assumed to be finite. As in the previous paragraph, there are $a_\lambda\in \mrm{c}_{00}(\bbGamma)$ such that $\vp_\lambda=\Theta^l(a_\lambda)\rest_{\CGd}$. Then $\Theta^l(a_\lambda)$ are normal, finite rank, CB maps in ${}_{\Ljd}\CB^\sigma(\Linfd)$ with $\|\Theta^l(a_\lambda)\|_{cb}=\|\vp_\lambda\|_{cb}$. Take $x\in \Linfd\setminus\{0\}, \omega\in \Ljd\setminus\{0\}$ and $\eps>0$.  Since products are linearly dense in $\Ljd$ (\cite[Section 3]{Crann}), we can find $\omega_k, \omega'_k\in \Ljd\,(1\le k \le K)$ such that $\|\omega-\sum_{k=1}^{K}\omega_k\star \omega'_k\|\le \tfrac{\eps}{2 \|x\| (1+ {}_{\Ljd}\Lambda_{cb}(\CGd)) }$. Furthermore, for any $k$ we have $\omega'_k\star x\in \CGd$ (\cite[Lemma 4.6]{DKV_ApproxLCQG}), hence there is $\lambda_0\in \Lambda$ such that 
\[
\|\Theta^l(a_\lambda)(\omega'_k\star x)-\omega'_k\star x\|\le
\tfrac{\eps}{2 K(1+\|\omega_k\|) } \quad(1\le k \le K, \lambda\ge \lambda_0).
\]
For $\lambda\ge \lambda_0$ we have
\[\begin{split}
&\quad\;
|\la \Theta^l(a_\lambda)(x)-x ,\omega \ra |\\
&\le 
\tfrac{\eps}{2} + \sum_{k=1}^{K} |\la  \Theta^l(a_\lambda)(x)-x,\omega_k\star\omega'_k \ra |=
\tfrac{\eps}{2} + \sum_{k=1}^{K} |\la \omega'_k\star  \Theta^l(a_\lambda)(x)-\omega'_k\star x,\omega_k \ra |\\
&=
\tfrac{\eps}{2} + \sum_{k=1}^{K} |\la  \Theta^l(a_\lambda)(\omega'_k\star x)-\omega'_k\star x,\omega_k \ra |\le 
\tfrac{\eps}{2} + \sum_{k=1}^{K}  
\|\Theta^l(a_\lambda)(\omega'_k\star x)-\omega'_k\star x\|\|\omega_k\|\le \eps.
\end{split}\]
This proves $\Theta^l(a_\lambda)(x)\xrightarrow[\lambda\in\Lambda]{} x$ weak$^*$ and consequently ${}_{\Ljd}\Lambda_{cb}(\Linfd)\le {}_{\Ljd} \Lambda_{cb}(\CGd)$. As above, a minor modification gives $\Ljd$-bimodule version.
\end{proof}

Because of Proposition \ref{prop1}, we will focus on ${}_{\Ljd} \Lambda_{cb}(\CGd)$ and ${}_{\Ljd}\Lambda_{cb,\Ljd}(\CGd)$. Let us now recall Cowling-Haagerup constant and its central variant for discrete quantum groups (\cite{Brannan, Averaging, Freslon}).

\begin{definition}\label{def1}
Let $\bbGamma$ be a discrete quantum group.
\begin{itemize}
\item The Cowling-Haagerup constant of $\bbGamma$, $\Lambda_{cb}(\bbGamma)\in\left[1,+\infty\right]$, is the infimum of all numbers $C\ge 1$ such that there is a net $(a_\lambda)_{\lambda\in \Lambda}$ in $\mrm{c}_{00}(\bbGamma)$  with $\|a_\lambda\|_{cb}\le C$ and $a_{\lambda}\xrightarrow[\lambda\in \Lambda]{} \I$ pointwise. If no such number exists, set $\Lambda_{cb}(\bbGamma)=+\infty$. If $\Lambda_{cb}(\bbGamma)<+\infty$, one says that $\bbGamma$ is weakly amenable.
\item The central Cowling-Haagerup constant $\mc{Z}\Lambda_{cb}(\bbGamma)\in\left[1,+\infty\right]$ is the infimum of all numbers $C\ge 1$ such that there is a net $(a_\lambda)_{\lambda\in \Lambda}$ in $\mc{Z} \mrm{c}_{00}(\bbGamma)$  with $\|a_\lambda\|_{cb}\le C$ and $a_{\lambda}\xrightarrow[\lambda\in \Lambda]{} \I$ pointwise. If no such number exists, set $\mc{Z} \Lambda_{cb}(\bbGamma)=+\infty$. If $\mc{Z}\Lambda_{cb}(\bbGamma)<+\infty$, one says that $\bbGamma$ is centrally weakly amenable.
\end{itemize}
\end{definition}

Similarly to Definition \ref{def2}, the infimum is actually attainable. Let us note that in the definition of $\Lambda_{cb}(\bbGamma)$ we could also consider more general nets $(a_\lambda)_{\lambda\in \Lambda}$ assumed only to be in the Fourier algebra $\A(\bbGamma)$. A standard approximation argument shows that both definitions are equivalent (see e.g.~\cite[Section 3.2]{Brannan}). Instead of speaking about pointwise convergence, one can require that $(a_\lambda)_{\lambda\in \Lambda}$ forms an approximate identity in the Fourier algebra $\A(\bbGamma)$. In this context, both conditions are equivalent. The following result provides a link between Cowling-Haagerup constant of a discrete quantum group $\bbGamma$ and the associated module $\mrm{C}(\wh\bbGamma)$. Its proof is quite standard, compare e.g.~\cite[Theorem 6.6]{Brannan}.

\begin{theorem}\label{thm3}
For any discrete quantum group $\bbGamma$
\[
\Lambda_{cb}(\bbGamma)={}_{\Ljd}\Lambda_{cb}(\mrm{C}(\wh\bbGamma))
\quad\textnormal{ and }\quad
\mc{Z}\Lambda_{cb}(\bbGamma)={}_{\Ljd}\Lambda_{cb,\LL^1(\wh\bbGamma)}(\mrm{C}(\wh\bbGamma)).
\]
\end{theorem}

\begin{proof}
Assume $\Lambda_{cb}(\bbGamma)<+\infty$ and let $(a_\lambda )_{\lambda\in \Lambda}$ be a net in $\mrm{c}_{00}(\bbGamma)$ such that $\|a_{\lambda}\|_{cb}\le \Lambda_{cb}(\bbGamma)$ and $a_{\lambda}\xrightarrow[\lambda\in\Lambda]{}\I$ pointwise. Since $\mrm{c}_{00}(\bbGamma)\subseteq \A(\bbGamma)$, we can find normal functionals $\omega_\lambda\in\Ljd$ such that $a_\lambda=\wh\lambda(\omega_\lambda)$. Then $\Theta^l(a_\lambda)=(\omega_\lambda\otimes \id)\wh\Delta$, consider $\vp_\lambda=\Theta^l(a_\lambda)\rest_{\CGd}$. This map is of finite rank, belongs to ${}_{\Ljd}\CB(\CGd)$ and has CB norm equal to $\|a_\lambda\|_{cb}$. Furthermore, since $\sup_{\lambda\in\Lambda}\|a_\lambda\|_{cb}<\infty$, to see that $\vp_\lambda\xrightarrow[\lambda\in\Lambda]{}\id$ in the point-norm topology of $\CGd$, it is enough to look at the dense subspace $\Pol(\wh\bbGamma)$. Since $a_{\lambda}\xrightarrow[\lambda\in\Lambda]{}\I$ pointwise, for any $\alpha\in\Irr(\wh\bbGamma),1\le i,j\le \dim(\alpha)$ we have $\omega_\lambda(U^{\alpha}_{i,j})\xrightarrow[\lambda\in\Lambda]{}\delta_{i,j}$ and consequently $\vp_\lambda(U^{\alpha}_{i,j})=\sum_{k=1}^{\dim(\alpha)} \omega_{\lambda}(U^{\alpha}_{i,k})U^{\alpha}_{k,j}$ converges in norm to $U^{\alpha}_{i,j}$. We conclude that ${}_{\Ljd}\Lambda_{cb}(\CGd)\le \Lambda_{cb}(\bbGamma)$.

Assume now that ${}_{\Ljd}\Lambda_{cb}(\CGd)<+\infty$ with the corresponding net $(\vp_\lambda)_{\lambda\in\Lambda}$ in ${}_{\Ljd}\CB(\CGd)$. As discussed in Section \ref{sec:preliminaries}, there is a multiplier $a_\lambda\in \M^l_{cb}(\A(\bbGamma))$ such that $\vp_\lambda=\Theta^l(a_\lambda)\rest_{\CGd}$, in particular $\|a_{\lambda}\|_{cb}=\|\vp_\lambda\|_{cb}$. Since $\vp_\lambda$ is of finite rank, we have in fact $a_\lambda\in \mrm{c}_{00}(\bbGamma)$. As the net $(\vp_\lambda)_{\lambda\in \Lambda}$ converges to $\id$ in the point-norm topology, we have $a_{\lambda}\xrightarrow[\lambda\in\Lambda]{}\I$ pointwise. This shows $\Lambda_{cb}(\bbGamma)\le{}_{\Ljd}\Lambda_{cb}(\CGd)$.

The central and bimodule variant is proved in a similar way, with slight modification. In the first direction, we additionally have $a_\lambda\in \mc{Z} \mrm{c}_{00}(\bbGamma)$, then $\Theta^l(a_\lambda)\rest_{\CGd}$ is a $\Ljd$-bimodule map giving ${}_{\Ljd}\Lambda_{cb,\Ljd}(\CGd)\le \mc{Z} \Lambda_{cb}(\bbGamma)$. Conversely since $\vp_\lambda$ is a map of bimodules, $a_\lambda$ is central.
\end{proof}

\begin{remark}
It is an interesting question whether ${}_{\Ljd}\Lambda_{cb}(\CGd)={}_{\Ljd}\Lambda_{cb,\Ljd}(\CGd)$ always holds, equivalently (by Theorem \ref{thm3}) whether Cowling-Haagerup constant of $\bbGamma$ is equal to its central variant $\Lambda_{cb}(\bbGamma)=\mc{Z}\Lambda_{cb}(\bbGamma)$. To the best of our knowledge, no counterexample is known. An analogous result for strong amenability is false (see e.g.~\cite[Theorem 7.6]{Averaging}).
\end{remark}

\section{Cowling-Haagerup constant of the product}\label{sec:main}

In this section we prove our main result: (central) Cowling-Haagerup constant of discrete quantum groups is multiplicative (Theorem \ref{thm2}). We will do this by establishing first an analogous result for modules $\CGd$ (Proposition \ref{prop2}) and then using Theorem \ref{thm3}. As mentioned in the introduction, our proof of Proposition \ref{prop2} is a modification of the proof of \cite[Theorem 12.3.13]{BrownOzawa} (see also Remark \ref{rem2}).\\

It will be convenient to work in the more general language of completely contractive Banach algebras and operator modules, see Section \ref{sec:preliminaries}. The next lemma is a bimodule generalisation of \cite[Lemma 12.3.16]{BrownOzawa}. Recall that any $A$-$B$-bimodule, is also a right $A^{op}\wh\otimes B$-module (see Section \ref{sec:preliminaries}).

\begin{lemma}\label{lemma5}
Let $A,B$ be completely contractive Banach algebras, $X$ an $A$-$B$-bimodule and $F\subseteq E\subseteq X$ finite dimensional submodules. Take $C\ge 1$. The following statements are equivalent:
\begin{enumerate}
\item there is $\vp\in \CB_{A^{op}\wh\otimes B}(X,E)$ such that $\vp(x)=x\,(x\in F)$ and $\|\vp\|_{\CB(X,E)}\le C$,
\item $|\kappa (u)|\le C \|q(u)\|$ for $u\in F\odot E^*$, where $\kappa\colon E\odot E^*\rightarrow \CC$ is the pairing map and $q\colon X\wh\otimes E^*\rightarrow X\wh{\otimes}_{A^{op}\wh{\otimes} B} E^*$ is the canonical quotient map.
\end{enumerate}
\end{lemma}

\begin{proof}
Assume that we have $\vp\in \CB_{A^{op}\wh\otimes B}(X,E)$ as in $(1)$, take $u\in F\odot E^*$ and write $u=\sum_{k=1}^{n} x_{k}\otimes \omega_{k}$ for some $x_{k}\in F,\omega_{k}\in E^*$. Then using identifications $\CB(F,E)=\CB(F,E^{**})\simeq (F\wh\otimes E^*)^*$ and $ \CB_{A^{op}\wh\otimes B}(X,E)\simeq (X\wh\otimes_{A^{op}\wh\otimes B} E^*)^*$ we calculate
\[\begin{split}
&\quad\;
|\kappa(u)|=\bigl|\sum_{k=1}^{n} \la \omega_{k}, x_{k}\ra \bigr|=
\bigl|\sum_{k=1}^{n} \la \omega_{k}, \vp(x_{k})\ra \bigr|
=|\la \vp , u \ra_{\CB(X,E), X\wh\otimes E^*}|\\
&=
|\la
\vp,q(u)
\ra_{\CB_{A^{op}\wh{\otimes} B}(X,E), X\wh{\otimes}_{A^{op}\wh{\otimes} B}E^*}|
\le 
C\|q(u)\|,
\end{split}\]
i.e.~$(2)$ holds.

Conversely, assume that $|\kappa(u)|\le C \|q(u)\|$ for all $u\in F\odot E^*$. Then the functional
\begin{equation}\label{eq2}
X\wh\otimes_{A^{op}\wh\otimes B} E^*\supseteq 
q(F\odot E^*)\ni q(u) \mapsto \kappa(u)\in \CC
\end{equation}
is well defined and has norm bounded by $C$. By Hahn-Banach theorem, we can find $\vp\in (X\wh\otimes_{A^{op}\wh\otimes B} E^*)^*\simeq \CB_{A^{op}\wh\otimes B}(X,E)$ which extends \eqref{eq2} and has norm $\le C$. For $x\in F, \omega\in E^*$ we have
\[
\la \omega, \vp(x)\ra = \la \vp, q(x\otimes \omega)\ra_{\CB_{A^{op}\wh\otimes B}(X,E),  X\wh\otimes_{A^{op}\wh\otimes B} E^*}=
\kappa(x\otimes \omega)=\la \omega,x\ra
\]
hence $\vp(x)=x$. 
\end{proof}

\begin{remark}\label{rem3}
For general completely contractive Banach algebras $A,B$ and bimodules $X,Y$ it can happen that ${}_A\CB_B(X,Y)\subsetneq\CB_{A^{op}\wh\otimes B}(X,Y)$ (consider the trivial module structure $ax=0,ay=0$ for $a\in A,x\in X,y\in Y$). On the other hand, in what follows we will be interested in the $\LL^1(\wh\bbGamma)$-bimodule structure on $\mrm{C}(\wh\bbGamma)$ for a discrete quantum group $\bbGamma$, or its sub-bimodules. In this situation we have equality ${}_{\LL^1(\wh\bbGamma)}\CB_{\LL^1(\wh\bbGamma)}(\mrm{C}(\wh\bbGamma))=\CB_{\LL^1(\wh\bbGamma)^{op}\wh\otimes \LL^1(\wh\bbGamma)}(\mrm{C}(\wh\bbGamma))$. Indeed, this easily follows from considering the dense subspace $\Pol(\wh\bbGamma)\subseteq \mrm{C}(\wh\bbGamma)$ and the following observation: for $x\in \Pol(\wh\bbGamma)$ there exists $\omega\in \LL^1(\wh\bbGamma)$ such that $\omega\star x=x\star \omega=x$.
\end{remark}

Next we establish several useful properties of the left $\Ljd$-module $\CGd$.

\begin{lemma}\label{lemma3}
Let $\bbGamma$ be a discrete quantum group and $C\ge 1$. The following conditions are equivalent:
\begin{enumerate}
\item ${}_{\Ljd}\Lambda_{cb}(\CGd)\le C$,
\item for every $\eps>0$ and finite $\emptyset\neq F\subseteq \Irr(\wh\bbGamma)$ there is a finite rank $\vp\in {}_{\Ljd}\CB(\CGd)$ such that $\|\vp(x) - x\|\le \eps \|x\| \,(x\in \Pol_F(\wh\bbGamma))$ and $\|\vp\|_{cb}\le  C$.
\end{enumerate}
\end{lemma}

\begin{proof}
$(1)\Rightarrow(2)$: take $\eps>0$ and finite $\emptyset\neq F\subseteq \Irr(\wh\bbGamma)$. Since $\Pol_F(\wh\bbGamma)$ is a finite dimensional normed space with basis $\{U^{\alpha}_{i,j}\,|\,\alpha\in F,\,1\le i,j\le \dim(\alpha)\}$ we can find $D>0$ so that
\[
\sum_{\alpha\in F}\sum_{i,j=1}^{\dim(\alpha)} |x^\alpha_{i,j}|\le D \|\sum_{\alpha\in F}\sum_{i,j=1}^{\dim(\alpha)} x^\alpha_{i,j}U^{\alpha}_{i,j}\|
\]
for all $\sum_{\alpha\in F}\sum_{i,j=1}^{\dim(\alpha)} x^\alpha_{i,j}U^{\alpha}_{i,j}\in \Pol_F(\wh\bbGamma)$. By $(1)$, there is $\vp\in {}_{\Ljd}\CB(\CGd)$ such that $\|\vp\|_{cb}\le C$ and
\[
\|\vp(U^{\alpha}_{i,j})-U^{\alpha}_{i,j}\|\le \tfrac{\eps}{D}\quad(\alpha\in F,1\le i,j\le \dim(\alpha)).
\]
Then for any $x=\sum_{\alpha\in F}\sum_{i,j=1}^{\dim(\alpha)} x^\alpha_{i,j}U^{\alpha}_{i,j}\in \Pol_F(\wh\bbGamma)$ we have
\[
\|\vp(x)-x\|\le \sum_{\alpha\in F} \sum_{i,j=1}^{\dim(\alpha)}|x^{\alpha}_{i,j}| \|\vp(U^{\alpha}_{i,j})-U^{\alpha}_{i,j}\|\le 
\sum_{\alpha\in F} \sum_{i,j=1}^{\dim(\alpha)}|x^{\alpha}_{i,j}| \tfrac{\eps}{D}\le \eps\|x\|.
\]
$(2)\Rightarrow (1)$: for $\eps>0$ and finite $\emptyset\neq F\subseteq \Irr(\wh\bbGamma)$, let $\vp_{\eps,F}\in {}_{\Ljd}\CB(\CGd)$ be the map from $(2)$. As $\|\vp_{\eps,F}\|_{cb}\le C$ for all $\eps,F$ and $\Pol(\wh\bbGamma)$ is norm dense in $\CGd$, it easily follows that net $(\vp_{\eps,F})_{(\eps,F)}$ indexed over $\eps\in \left]0,1\right[$ and finite $\emptyset\neq F\subseteq \Irr(\wh\bbGamma)$ gives ${}_{\Ljd}\Lambda_{cb}(\CGd)\le C$.
\end{proof}

There is also a natural analog of Lemma \ref{lemma3} for $\Ljd$-bimodule $\CGd$. Recall that for finite $\emptyset\neq F\subseteq \Irr(\wh\bbGamma)$, $p_F\in \mrm{c}_{00}(\bbGamma)\subseteq \A(\bbGamma)$ is the central projection $p_F=\sum_{\alpha\in F}p_\alpha$.

\begin{lemma}\label{lemma1}
Let $\bbGamma$ be a discrete quantum group and $\eps>0$. For a finite set $\emptyset\neq F\subseteq \Irr(\wh\bbGamma)$,
\[
1\le \|p_F\|_{\A(\bbGamma)}\le\sqrt{\sum_{\alpha\in F}\dim_q(\alpha)^2 }.
\]
\end{lemma} 

\begin{proof}
Recall that $p_F=\wh\lambda(\omega_F)$ (equation \eqref{eq10}). Let $\|x\|_2=h(x^* x)^{1/2}\,(x\in\LL^{\infty}(\wh\bbGamma))$ be the $2$-norm on $\Linfd$. Using orthogonality relations (\cite[Theorem 1.4.3]{NeshveyevTuset}) we see
\[
\begin{split}
\|p_F\|_{\A(\bbGamma)}^2=\|\omega_F\|^2&\le 
\bigl\| 
\sum_{\alpha\in F}\sum_{i=1}^{\dim(\alpha)} \dim_q(\alpha) \uprho_{\alpha,i} U^{\alpha }_{i,i}\bigr\|_2^2=
\sum_{\alpha\in F}\sum_{i=1}^{\dim(\alpha)}
\dim_q(\alpha)^2 {\uprho_{\alpha,i}}^2 \|U^{\alpha }_{i,i}\|_2^2\\
&=
\sum_{\alpha\in F}\sum_{i=1}^{\dim(\alpha)}
\dim_q(\alpha)^2 {\uprho_{\alpha,i}}^2 \tfrac{1}{\dim_q(\alpha)\uprho_{\alpha,i}}=\sum_{\alpha\in F}\dim_q(\alpha)^2.
\end{split}
\]
For the lower bound, choose $\alpha\in F$ and let $\chi_\alpha=\sum_{i=1}^{\dim(\alpha)} U^{\alpha}_{i,i}$ be character of $\alpha$. Then $\|\chi_\alpha\|\le \dim(\alpha)$ and 
\[
\|p_F\|_{\A(\bbGamma)}=\|\omega_F\|\ge \bigl|\omega_F(\tfrac{\chi_\alpha}{\|\chi_\alpha\|})\bigr|=
\tfrac{1}{\|\chi_\alpha\|}
\bigl|\sum_{i=1}^{\dim(\alpha)}
\dim_q(\alpha)\uprho_{\alpha,i} h(U^{\alpha *}_{i,i}\chi_\alpha)\bigr|=
\tfrac{\dim(\alpha)}{\|\chi_\alpha\|}\ge 1.
\]
\end{proof}

The next lemma shows intuitively that one can correct an almost equality $a\approx \I$ over a finite set $F\subseteq \Irr(\wh\bbGamma)$ to an actual equality, with an error over which we have precise control.

\begin{lemma}\label{lemma2}
Let $\bbGamma$ be a discrete quantum group, $\eps>0$, $\emptyset\neq F\subseteq \Irr(\wh\bbGamma)$ finite set and $a\in \M^l_{cb}(\A(\bbGamma))$. Assume that $\|\Theta^l(a)(x)-x\|\le \eps \|x\|$ for $x\in \Pol_F(\wh\bbGamma)$. Then there is $\tilde{a}\in \M^l_{cb}(\A(\bbGamma))$ such that $\Theta^l(\tilde{a})(x)=x$ for $x\in \Pol_F(\wh\bbGamma)$, $\tilde{a}-a\in \mrm{c}_{00}(\bbGamma)$ and $\|\tilde{a}-a\|_{\A(\bbGamma)}\le \eps \sum_{\alpha\in F}\dim_q(\alpha)^2$. If $a\in \mc{Z} \M^l_{cb}(\A(\bbGamma))$, then we can take $\tilde{a}\in \mc{Z}\M^l_{cb}(\A(\bbGamma))$.
\end{lemma}

\begin{proof}
Write $a=(a_\alpha)_{\alpha\in\Irr(\wh\bbGamma)}$. Define $b=\sum_{\alpha\in F} (p_\alpha - a_\alpha)\in \mrm{c}_{00}(\bbGamma)$ and $\tilde{a}=a+b$. Since $\tilde{a}_\alpha=p_\alpha\,(\alpha\in F)$, we have $\Theta^l(\tilde{a})(x)=x$ for $x\in \Pol_F(\wh\bbGamma)$. Furthermore
\[
\tilde{a}-a=b=
\sum_{\alpha\in F} (p_\alpha - a_\alpha)=
(\I - a)\sum_{\alpha\in F} p_\alpha=
(\I-a)p_F=(\I-a)\wh\lambda(\omega_F)=
\wh\lambda\bigl( \Theta^l(\I-a)_*(\omega_F)\bigr).
\]
Consequently, using the facts that $\Theta^l(p_F)_*(\omega_F)=\omega_F$, $\Theta^l(p_F)(x)\in \Pol_F(\wh\bbGamma)$ for $x\in \Linfd$ and $p_F$ is central
\[\begin{split}
&\quad\;
\|\tilde{a}-a\|_{\A(\bbGamma)}=
\|\Theta^l(\I - a)_*( \omega_F)\|=
\sup_{x\in \LL^{\infty}(\wh\bbGamma),\|x\|=1}\!
\bigl| \la x,\Theta^l(\I-a)_*(\omega_F) \ra \bigr|\\
&=
\sup_{x\in \LL^{\infty}(\wh\bbGamma),\|x\|=1}\!
\bigl| \la x,\Theta^l(\I-a)_*\Theta^l(p_F)_*(\omega_F) \ra \bigr|=
\sup_{x\in \LL^{\infty}(\wh\bbGamma),\|x\|=1}\!
\bigl| \la \Theta^l(\I-a) \Theta^l(p_F)(x),\omega_F \ra \bigr|\\
&\le 
\sup_{x\in \LL^{\infty}(\wh\bbGamma),\|x\|=1}\!
\| \Theta^l(p_F)(x)- \Theta^l(a) (\Theta^l(p_F)(x))\|\|\omega_F \|\le
\sup_{x\in \LL^{\infty}(\wh\bbGamma),\|x\|=1}\!
\eps  \|\Theta^l(p_F)(x)\|\|\omega_F\|\\
&=
\eps \|\omega_F\| \|\Theta^l(p_F)\|\le 
\eps \|p_F\|_{\A(\bbGamma)} \|p_F\|_{cb}\le 
\eps \|p_F\|_{\A(\bbGamma)}^2,
\end{split}\]
hence the first claim follows from Lemma \ref{lemma1}. If $a$ is central, then so is $b$ and consequently $\tilde{a}$.
\end{proof}
 Let us remark that using \cite[Corollary 2.2.4]{EffrosRuan} one can obtain a better bound for $\|\tilde{a}-a\|_{cb}$ -- we will however not need this. Our main result, in the language of modules, is the following.

\begin{proposition}\label{prop2}
Let $\bbGamma_1,\bbGamma_2$ be discrete quantum groups and $\bbGamma=\bbGamma_1\times \bbGamma_2$ their product. Then
\begin{equation}\label{eq1}
{}_{\Ljd}\Lambda_{cb}(\CGd)=
{}_{\LL^1(\wh\bbGamma_1)}\Lambda_{cb}(\mrm{C}(\wh\bbGamma_1))\;
{}_{\LL^1(\wh\bbGamma_2)}\Lambda_{cb}(\mrm{C}(\wh\bbGamma_2))
\end{equation}
and
\begin{equation}\label{eq6}
{}_{\Ljd}\Lambda_{cb,\Ljd}(\CGd)=
{}_{\LL^1(\wh\bbGamma_1)}\Lambda_{cb,\LL^1(\wh\bbGamma_2)}(\mrm{C}(\wh\bbGamma_1))\;
{}_{\LL^1(\wh\bbGamma_2)}\Lambda_{cb,\LL^1(\wh\bbGamma_2)}(\mrm{C}(\wh\bbGamma_2)).
\end{equation}
\end{proposition}

\begin{proof}
The easier inequality $\le$ was already established in \cite[Proposition 3.2]{FreslonPermanence} (after conjunction with Theorem \ref{thm3}), let us give an essentially equivalent argument for the convenience of the reader. Recall that $\CGd=\mrm{C}(\wh\bbGamma_1)\otimes \mrm{C}(\wh\bbGamma_2)$ as \cst-algebras and $\Ljd=\LL^1(\wh\bbGamma_1)\wh\otimes \LL^1(\wh\bbGamma_2)$ as completely contractive Banach algebras. It is enough to assume that both ${}_{\LL^1(\wh\bbGamma_1)}\Lambda_{cb}(\mrm{C}(\wh\bbGamma_1))$ and ${}_{\LL^1(\wh\bbGamma_2)}\Lambda_{cb}(\mrm{C}(\wh\bbGamma_2))$ are finite, let $(\vp_\lambda)_{\lambda\in \Lambda}$ and $(\psi_\mu)_{\mu\in \Sigma}$ be the corresponding maps. Then we can construct new net $(\vp_\lambda\otimes \psi_\mu)_{(\lambda,\mu)\in \Lambda\times \Sigma}$ of finite rank maps in ${}_{\LL^1(\wh\bbGamma_1)\wh\otimes \LL^1(\wh\bbGamma_2)}\CB(\mrm{C}(\wh\bbGamma_1)\otimes \mrm{C}(\wh\bbGamma_2))$ (\cite[Proposition 8.1.5]{EffrosRuan}). For any $\lambda,\mu$ we have $\|\vp_\lambda\otimes \psi_\mu\|_{cb}\le {}_{\LL^1(\wh\bbGamma_1)}\Lambda_{cb}(\mrm{C}(\wh\bbGamma_1))\, {}_{\LL^1(\wh\bbGamma_2)}\Lambda_{cb}(\mrm{C}(\wh\bbGamma_2))$ and clearly $\vp_\lambda\otimes \psi_\mu\xrightarrow[(\lambda,\mu)\in \Lambda\times \Sigma]{}\id$ on a norm dense set $\mrm{C}(\wh\bbGamma_1)\odot\mrm{C}(\wh\bbGamma_2)\subseteq \CGd$. Consequently $(\vp_\lambda\otimes \psi_\mu)_{(\lambda,\mu)\in\Lambda\times\Sigma}$ converges to $\id$ in the point-norm topology. This allows us to conclude inequality $\le$ in \eqref{eq1}. An analogous reasoning gives inequality $\le $ in \eqref{eq6}: the only difference is that if $\vp_\lambda$ and $\psi_\mu$ are bimodule maps, then so is $\vp_\lambda\otimes \psi_\mu$.\\

Let us now prove the converse inequalities; we will treat both cases at the same time. Assume by contradiction that \eqref{eq1} or \eqref{eq6} does not hold, i.e.
\[
{}_{\LL^1(\wh\bbGamma)}\Lambda_{cb}(\CGd)<
{}_{\LL^1(\wh\bbGamma_1)}\Lambda_{cb}(\mrm{C}(\wh\bbGamma_1))
\;{}_{\LL^1(\wh\bbGamma_2)}\Lambda_{cb}(\mrm{C}(\wh\bbGamma_2))\]
or
\[
{}_{\LL^1(\wh\bbGamma)}\Lambda_{cb,\Ljd}(\CGd)<
{}_{\LL^1(\wh\bbGamma_1)}\Lambda_{cb,\LL^1(\wh\bbGamma_1)}(\mrm{C}(\wh\bbGamma_1))
\;{}_{\LL^1(\wh\bbGamma_2)}\Lambda_{cb,\LL^1(\wh\bbGamma_2)}(\mrm{C}(\wh\bbGamma_2)).
\] Then we can choose positive constants $C_1,C_2$ such that
\begin{align}
{}_{\LL^1(\wh\bbGamma)}\Lambda_{cb}(\CGd)&<
C_1C_2,\label{eq5a}\\
1\le C_1 <{}_{\LL^1(\wh\bbGamma_1)}\Lambda_{cb}(\mrm{C}(\wh\bbGamma_1))&,\quad
1\le C_2<{}_{\LL^1(\wh\bbGamma_2)}\Lambda_{cb}(\mrm{C}(\wh\bbGamma_2))\label{eq5b}
\end{align}
in the left module case and
\begin{align}
{}_{\LL^1(\wh\bbGamma)}\Lambda_{cb,\Ljd}(\CGd)&<
C_1C_2,\label{eq5c}\\
1\le C_1 <{}_{\LL^1(\wh\bbGamma_1)}\Lambda_{cb,\LL^1(\wh\bbGamma_1)}(\mrm{C}(\wh\bbGamma_1))&,\quad
1\le C_2<{}_{\LL^1(\wh\bbGamma_2)}\Lambda_{cb,\LL^1(\wh\bbGamma_2)}(\mrm{C}(\wh\bbGamma_2))\label{eq5d}
\end{align}
in the bimodule case.

In order to easier work with both cases at the same time, it will be convenient to reformulate the situation slightly. As discussed in Section \ref{sec:preliminaries}, left $\Ljd$-module structure on $\CGd$ gives us right $\Ljd^{op}$-module structure. Similarly, $\Ljd$-bimodule structure can also be encoded as right $\LL^1(\wh\bbGamma)^{op}\wh\otimes \Ljd$-module structure. Thus from now on, let $A$ be equal to $\Ljd^{op}$ or $\Ljd^{op}\wh\otimes \Ljd$, and consider $\CGd$ as a right $A$-module. Similarly for quantum groups $\bbGamma_1,\bbGamma_2$ consider $\mrm{C}(\wh\bbGamma_k)$ as a right $A_k$-module, where $A_k=\LL^1(\wh\bbGamma_k)^{op}$ or $A_k=\LL^1(\wh\bbGamma_k)^{op}\wh\otimes \LL^1(\wh\bbGamma_k)$.\\

First we use ``negative'' \eqref{eq5b} (or \eqref{eq5d}). Fix $k\in\{1,2\}$ and use Lemma \ref{lemma3} (or its bimodule version) to find $\eps_k>0$ and a finite set $\emptyset\neq F_k\subseteq \Irr(\wh\bbGamma_k)$ such that for all finite rank maps $\vp\in \CB_{A_k}(\mrm{C}(\wh\bbGamma_k))$ with $\|\vp\|_{\CB(\mrm{C}(\wh\bbGamma_k))}\le C_k$ there is $x\in \Pol_{F_k}(\wh\bbGamma_k)$ with $\|\vp(x)-x\|>\eps_k \|x\|$. In particular
\begin{equation}\label{eq8}
\|\vp\rest_{\Pol_{F_k}(\wh\bbGamma_k)}-\id\|_{\CB(\Pol_{F_k}(\wh\bbGamma_k), \mrm{C}(\wh\bbGamma_k))}>\eps_k.
\end{equation}

Now we use ``positive'' \eqref{eq5a} (or \eqref{eq5c}). Define $F=F_1\boxtimes F_2 \subseteq \Irr(\wh\bbGamma)$ and choose small $\delta>0$ such that
\[
{}_{\LL^1(\wh\bbGamma)}\Lambda_{cb}(\mrm{C}(\wh\bbGamma))<(1-\delta)C_1 C_2<C_1 C_2 \quad\textnormal{or}\quad
{}_{\LL^1(\wh\bbGamma)}\Lambda_{cb,\Ljd}(\mrm{C}(\wh\bbGamma))<(1-\delta)C_1 C_2<C_1 C_2
\]
depending on the version we are considering. Next set $\eps= \tfrac{\delta C_1 C_2}{\sum_{\alpha\in F} \dim_q(\alpha)^2}>0$. For this $\eps$ and $F$, by Lemma \ref{lemma3} we can find finite rank $\vp\in \CB_{A}(\mrm{C}(\wh\bbGamma))$ with $\|\vp\|_{\CB(\CGd)} \le (1-\delta)C_1 C_2$ and $\|\vp(x)-x\| \le \eps \|x\|$ for $x\in \Pol_F(\wh\bbGamma)$. Since $\vp$ is a right $A$-module map, it corresponds to $a\in \mrm{c}_{00}(\bbGamma)$ (or $a\in \mc{Z} \mrm{c}_{00}(\bbGamma)$) via $\vp=\Theta^l(a)\rest_{\CGd}$. Choose $\tilde{a}\in \mrm{c}_{00}(\bbGamma)$ (or $\tilde{a}\in \mc{Z}\mrm{c}_{00}(\bbGamma)$) using Lemma \ref{lemma2}, so that $\Theta^l(\tilde{a})=\id $ on $\Pol_F(\wh\bbGamma)$ and since the CB norm is majorised by Fourier algebra norm
\[
\|\Theta^l(\tilde{a})\rest_{\CGd}\|_{\CB(\CGd)}=\|\tilde{a}\|_{cb}\le 
\|a\|_{cb}+
\|\tilde{a} - a\|_{cb}\le 
(1-\delta) C_1 C_2 + \eps \sum_{\alpha\in F}\dim_q(\alpha)^2= C_1 C_2.
\]
$\Theta^l(\tilde{a})\rest_{\CGd}$ is a finite rank right $A$-module map, hence it has image in $\Pol_E(\wh\bbGamma)$ for some finite $E\subseteq \Irr(\wh\bbGamma)$. By enlarging $E$ if needed, we can assume $E=E_1\boxtimes E_2$ for finite $\emptyset\neq E_k\subseteq \Irr(\wh\bbGamma_k)$ with $F_k\subseteq E_k$. Existence of $\Theta^l(\tilde{a})\rest_{\CGd}$ shows that point $(1)$ of Lemma \ref{lemma5} holds (for modules $\Pol_F(\wh\bbGamma)\subseteq \Pol_E(\wh\bbGamma)$ and constant $C_1 C_2$), consequently $(2)$ of this lemma gives
\begin{equation}\label{eq3}
|\kappa(u)|\le 
C_1 C_2 \|q(u)\|
\end{equation}
 for $u\in \Pol_F(\wh\bbGamma)\odot \Pol_E(\wh\bbGamma)^*$. Here $q$ is the quotient map $\mrm{C}(\wh\bbGamma)\wh\otimes \Pol_E(\wh\bbGamma)^*\rightarrow \mrm{C}(\wh\bbGamma)\wh\otimes_A\Pol_E(\wh\bbGamma)^*$.\\
 
Next we go back to the reasoning concerning $\bbGamma_k$'s. Consider finite dimensional right $A_k$-submodules $\Pol_{F_k}(\wh\bbGamma_k)\subseteq \Pol_{E_k}(\wh\bbGamma_k)$ of $\mrm{C}(\wh\bbGamma_k)$ and numbers $C_k$. We will denote this action and its dual by $ x \vartriangleleft f, f\vartriangleright \omega\;(x\in \mrm{C}(\wh\bbGamma_k),\omega\in \Pol_{E_k}(\wh\bbGamma_k)^*,f\in A_k)$ to avoid confusion. By the reasoning above (inequality \eqref{eq8}), point $(1)$ in Lemma \ref{lemma5} does not hold, and there is $u_k\in \Pol_{F_k}(\wh\bbGamma_k)\odot \Pol_{E_k}(\wh\bbGamma_k)^*$ such that
\begin{equation}\label{eq12}
|\kappa(u_k)|>C_k\|q(u_k)\|.
\end{equation}
We claim that $q(u_k)\neq 0$. To see this, we need to introduce an auxilliary bounded functional. First, observe that we can understand $\Theta^l(p_{E_k})\rest_{\mrm{C}(\wh\bbGamma_k)}$ as a CB map $\mrm{C}(\wh\bbGamma_k)\rightarrow \Pol_{E_k}(\wh\bbGamma_k)$. Next consider its dual map and define $\rho$ to be the composition
\[
\rho
\colon \mrm{C}(\wh\bbGamma_k)\wh\otimes \Pol_{E_k}(\wh\bbGamma_k)^*\xrightarrow[]{
\id\otimes (\Theta^l(p_{E_k})\rest_{\mrm{C}(\wh\bbGamma_k)})^* }
\mrm{C}(\wh\bbGamma_k)\wh\otimes \mrm{C}(\wh\bbGamma_k)^*\xrightarrow[]{\kappa}\CC.
\]
Let us write
\begin{equation}\label{eq11}
u_k=\sum_{i=1}^{N_k} x_{k,i}\otimes \omega_{k,i}\quad\textnormal { for } x_{k,i}\in \Pol_{F_k}(\wh\bbGamma_k)\subseteq \mrm{C}(\wh\bbGamma_k)\;\textnormal { and }\; \omega_{k,i}\in \Pol_{E_k}(\wh\bbGamma_k)^*
\end{equation}
and observe
\begin{equation}\label{eq13}
\la \rho,u_k\ra=
\sum_{i=1}^{N_k} \la \rho,x_{k,i}\otimes\omega_{k,i}\ra=
\sum_{i=1}^{N_k} \la \omega_{k,i},\Theta^l(p_{E_k})(x_{k,i})\ra=
\sum_{i=1}^{N_k} \la \omega_{k,i},x_{k,i}\ra=
\kappa(u_k).
\end{equation}
Assume by contradiction that $q(u_k)=0$, then
\[
u_k\in \ov{\lin}\{x \vartriangleleft f \otimes \omega - x\otimes f \vartriangleright \omega\,|\, x\in \mrm{C}(\wh\bbGamma_k),f\in A_k,\omega\in \Pol_{E_k}(\wh\bbGamma_k)^*\}\subseteq 
\mrm{C}(\wh\bbGamma_k)\wh\otimes \Pol_{E_k}(\wh\bbGamma_k)^*.
\]
Since
\[\begin{split}
\la \rho,
x \vartriangleleft f \otimes \omega - x\otimes f \vartriangleright \omega \ra &=
\la \omega , 
\Theta^l(p_{E_k})( x\vartriangleleft f ) \ra -
\la f\vartriangleright \omega , \Theta^l(p_{E_k})(x)\ra \\
&=
\la \omega , 
\Theta^l(p_{E_k})( x )\vartriangleleft f \ra -
\la f\vartriangleright \omega , \Theta^l(p_{E_k})(x)\ra 
=0
\end{split}\]
for $x\in \mrm{C}(\wh\bbGamma_k),f\in A_k,\omega\in \Pol_{E_k}(\wh\bbGamma_k)^*$, we have $\la \rho,u_k\ra=0$ by continuity of $\rho$. This contradicts \eqref{eq12} and \eqref{eq13}, consequently $q(u_k)\neq 0$.

Let us introduce shuffling map (cf.~\cite[Lemma 12.3.14]{BrownOzawa}) 
\[
\bigl(\mrm{C}(\wh\bbGamma_1)\wh\otimes \Pol_{E_1}(\wh\bbGamma_1)^*\bigr)\times 
\bigl(\mrm{C}(\wh\bbGamma_2)\wh\otimes  \Pol_{E_2}(\wh\bbGamma_2)^*\bigr)\ni (v_1,v_2)\mapsto v_1\times v_2 \in 
\mrm{C}(\wh\bbGamma)\wh\otimes \Pol_E(\wh\bbGamma)^*
\]
given by the bilinear extention of 
\[
(x_1\otimes \omega_1)\times (x_2\otimes \omega_2)=x_1\otimes x_2 \otimes \omega_1\otimes \omega_2
\]
(it is well defined as $E=E_1\boxtimes E_2$ is finite, hence we can identify completely isometrically $\Pol_E(\wh\bbGamma)=\Pol_{E_1}(\wh\bbGamma_1)\wc{\otimes} \Pol_{E_2}(\wh\bbGamma_2)$, where $\wc{\otimes}$ is the injective operator space tensor product \cite[Section 8]{EffrosRuan}). According to \cite[Lemma 12.3.14]{BrownOzawa} we have $\|v_1\times v_2\|\le \|v_1\| \|v_2\|$ for any $v_k\in \mrm{C}(\wh\bbGamma_k)\wh\otimes \Pol_{E_k}(\wh\bbGamma_k)^*$. Consider
\[
u=u_1\times u_2\in \Pol_F(\wh\bbGamma)\odot \Pol_E(\wh\bbGamma)^*\subseteq \mrm{C}(\wh\bbGamma)\wh\otimes \Pol_E(\wh\bbGamma)^*.
\]
We will use this element to obtain a contradiction. Using \eqref{eq11} we have $u=\sum_{i=1}^{N_1} \sum_{j=1}^{N_2}
x_{1,i}\otimes x_{2,j}\otimes \omega_{1,i}\otimes \omega_{2,j}$ and consequently
\[
\kappa(u)=
\sum_{i=1}^{N_1}\sum_{j=1}^{N_2}
\la \omega_{1,i}\otimes \omega_{2,j},x_{1,i}\otimes x_{2,j}\ra =
\sum_{i=1}^{N_1}\sum_{j=1}^{N_2}
\la \omega_{1,i},x_{1,i}\ra \,\la \omega_{2,j}, x_{2,j}\ra =
\kappa(u_1)\kappa(u_2)
\]
and
\begin{equation}\label{eq7}
|\kappa(u)|=
|\kappa(u_1)|\,|\kappa(u_2)|>
C_1 C_2\|q(u_1)\|\,\|q(u_2)\|.
\end{equation}
by \eqref{eq12}. Next we need to get a hold on the norm $\|q(u)\|$, which is the norm in the quotient space $\mrm{C}(\wh\bbGamma)\wh\otimes_{A} \Pol_E(\wh\bbGamma)^*=
(\mrm{C}(\wh\bbGamma)\wh\otimes \Pol_E(\wh\bbGamma)^*)/\ker (q)$. Fix an arbitrary $\eps_0>0$. For $k\in\{1,2\}$ we can choose
\[\begin{split}
n_k&\in \ker (q\colon \mrm{C}(\wh\bbGamma_k)\wh\otimes \Pol_{E_k}(\wh\bbGamma_k)^*\rightarrow \mrm{C}(\wh\bbGamma_k)\wh\otimes_{A_k}\Pol_{E_k}(\wh\bbGamma_k)^*)\\
&=\ov{\lin}\{ n\vartriangleleft f\otimes \nu - n\otimes f\vartriangleright \nu \,|\,n\in\mrm{C}(\wh\bbGamma_k),f\in A_k,\nu\in \Pol_{E_k}(\wh\bbGamma_k)^*\}
\end{split}
\]
such that $\|u_k+n_k\|-\eps_0 \le  \|q(u_k)\|\le \|u_k+n_k\|$. We can write
\[
n_k=\lim_{j\to\infty} \sum_{l=1}^{L_k^j} (n_{k,l}^j \vartriangleleft f_{k,l}^j\otimes \nu_{k,l}^j- n_{k,l}^j\otimes f_{k,l}^j\vartriangleright \nu_{k,l}^j)
\]
for some $n_{k,l}^j\in \mrm{C}(\wh\bbGamma_k), f_{k,l}^j\in A_k, \nu_{k,l}^j\in \Pol_{E_k}(\wh\bbGamma_k)^*$. Then
\[\begin{split}
&\quad\; q(u_1\times n_2)=
\lim_{j\to\infty}
\sum_{i=1}^{N_1} \sum_{l=1}^{L_2^j}
q\bigl(
(x_{1,i}\otimes \omega_{1,i})\times (
n_{2,l}^j\vartriangleleft f_{2,l}^j\otimes \nu_{2,l}^j - 
n_{2,l}^j\otimes f_{2,l}^j\vartriangleright \nu_{2,l}^j
)\bigr)\\
&=
\lim_{j\to\infty}
\sum_{i=1}^{N_1} \sum_{l=1}^{L_2^j}
q\bigl(
x_{1,i}\otimes 
n_{2,l}^j\vartriangleleft f_{2,l}^j\otimes 
\omega_{1,i}\otimes \nu_{2,l}^j
-
x_{1,i}\otimes 
n_{2,l}^j\otimes 
\omega_{1,i}\otimes f_{2,l}^j\vartriangleright  \nu_{2,l}^j
\bigr)\\
&=
\lim_{j\to\infty}
\sum_{i=1}^{N_1} \sum_{l=1}^{L_2^j}
q\bigl(
(x_{1,i}\otimes 
n_{2,l}^j)\vartriangleleft (\omega\otimes f_{2,l}^j)\otimes 
(\omega_{1,i}\otimes \nu_{2,l}^j )
-
(x_{1,i}\otimes 
n_{2,l}^j)\otimes 
(\omega\otimes  f_{2,l}^j)\vartriangleright (\omega_{1,i}\otimes  \nu_{2,l}^j )
\bigr)
\end{split}\]
where $\omega\in \LL^1(\wh\bbGamma_1)$ (or $\omega\in \LL^1(\wh\bbGamma_1)\wh\otimes\LL^1(\wh\bbGamma_1)$) is any normal functional which on $\Pol_{E_1}(\wh\bbGamma_1)$ acts as the counit -- so $x_{1,i}\vartriangleleft \omega =x_{1,i}$ and $\omega\vartriangleright \omega_{1,i}=\omega_{1,i}$. Such functional can be easily constructed using orthogonality relations \cite[Theorem 1.4.3]{NeshveyevTuset}, for example we can take $\omega=\omega_{E_1}$ (or $\omega=\omega_{E_1}\otimes\omega_{E_1}$). It follows that $q(u_1\times n_2)=0$. Similarly we check $q(n_1\times u_2)=0$ and $q(n_1\times n_2)=0$. Consequently
\[q(u)=q(u_1\times u_2)=
q(u_1 \times u_2+
n_1 \times u_2+
u_1 \times n_2+
n_1 \times n_2)\]
so
\[\begin{split}
&\quad\;
\|q(u)\|
\le 
\|u_1 \times u_2+
n_1 \times u_2+
u_1 \times n_2+
n_1 \times n_2
\|=
\|(u_1+n_1)\times (u_2+ n_2)\|\\
&\le  \|u_1+n_2\|\,\|u_2+n_2\|\le (\|q(u_1)\|+\eps_0)(\|q(u_2)\|+\eps_0).
\end{split}\]
Since $\eps_0>0$ was arbitrary, we conclude $\|q(u)\|\le \|q(u_1)\| \|q(u_2)\|$. Combining this with inequalities \eqref{eq3} and \eqref{eq7} we get
\[
C_1 C_2 \|q(u_1)\|\,\|q(u_2)\|<
C_1 C_2 \|q(u_1)\| \,\|q(u_2)\|,
\]
and as $q(u_1)\neq 0, q(u_2)\neq 0$ this gives a contradicition.
\end{proof}

\begin{remark}\label{rem2}
We have formulated and proved Proposition \ref{prop2} only for modules of the form $\CGd$ because of two reasons. First, in the case of $\CGd$ there is a canonical dense submodule $\Pol(\wh\bbGamma)$ whose finite dimensional subspaces give a wealth of finite dimensional submodules. Another reason is that for any finite $\emptyset\neq E\subseteq \Irr(\wh\bbGamma)$ one can find $\omega\in \Ljd$ which acts as the identity on $\Pol_E(\wh\bbGamma)$. This ``local unitality'' property was used to obtain bound $\|q(u_1\times u_2)\|\le \|q(u_1)\|\,\|q(u_2)\|$.
\end{remark}

\begin{theorem}\label{thm2}
Let $\bbGamma_1,\bbGamma_2$ be discrete quantum groups and $\bbGamma=\bbGamma_1\times \bbGamma_2$ their product. Then
\[
\Lambda_{cb}(\bbGamma)=
\Lambda_{cb}(\bbGamma_1)\,\Lambda_{cb}(\bbGamma_2)
\quad
\textnormal{and}\quad
\mc{Z}\Lambda_{cb}(\bbGamma)=
\mc{Z} \Lambda_{cb}(\bbGamma_1)\,\mc{Z}\Lambda_{cb}(\bbGamma_2).
\]
\end{theorem}

\begin{proof}
This result is an immediate consequence of Proposition \ref{prop2} and Theorem \ref{thm3}.
\end{proof}

As a corollary, we extend this result to infinite direct sums. Let $(\bbGamma_i)_{i\in I}$ be a non-empty family of discrete quantum groups. Then one can define product $\prod_{i\in I}\wh\bbGamma_i$, which is a compact quantum group (\cite{Wang}, see also \cite[Section 7.2]{DKV_ApproxLCQG}). We will denote its discrete dual by $\bigoplus_{i\in I}\bbGamma_i$ and call it the direct sum of family $(\bbGamma_i)_{i\in I}$ (the name and notation is inspired by the classical case where $\prod_{i\in I} \Gamma_i$ is larger than $\bigoplus_{i\in I} \Gamma_i$ whenever $|I|=\infty$ and $|\Gamma_i|\ge 2$).

\begin{corollary}\label{cor1}
Let $(\bbGamma_i)_{i\in I}$ be a non-empty family of discrete quantum groups and let $\bbGamma=\bigoplus_{i\in I} \bbGamma_i$ be their direct sum. Then
\begin{equation}\label{eq9}
\Lambda_{cb}(\bbGamma)=
\prod_{i\in I} \Lambda_{cb}(\bbGamma_i)
\quad
\textnormal{and}\quad
\mc{Z}\Lambda_{cb}(\bbGamma)=
\prod_{i\in I}\mc{Z} \Lambda_{cb}(\bbGamma_i).
\end{equation}
\end{corollary}

\begin{proof}
If $I$ is finite, then the claim follows immediately from Theorem \ref{thm2}; assume that $|I|=\infty$. Discrete quantum group $\bbGamma$ is the direct limit of system $(\oplus_{i\in F} \bbGamma_i)_{F}$ indexed by finite non-empty subsets $F\subseteq I$ with the canonical injective maps $\mrm{C}(\prod_{i\in F}\wh\bbGamma_i)\ni x \mapsto x\otimes (\otimes_{i\in F'\setminus F}\I_i)\in \mrm{C}(\prod_{i\in F'} \wh\bbGamma_i)$ for $F\subseteq F'$. Using Theorem \ref{thm2} and \cite[Proposition 3.6]{FreslonPermanence} we have
\[
\Lambda_{cb}(\bbGamma)=\sup_F \Lambda_{cb} \bigl(\bigoplus_{i\in F} \bbGamma_i\bigr)=
\sup_F \prod_{i\in F} \Lambda_{cb} ( \bbGamma_i)=
\prod_{i\in I}\Lambda_{cb}(\bbGamma_i)
\]
(recall $\Lambda_{cb}(\bbGamma_i)\ge 1$). One easily sees that \cite[Proposition 3.6]{FreslonPermanence} holds also for the central Cowling-Haagerup constant, which gives the second equality in \eqref{eq9}.

Alternatively, one can prove both equalities \eqref{eq9} as follows. Lower bounds follow from Theorem \ref{thm2} and decomposition $\bigoplus_{i\in I} \bbGamma_i=\bigl(\bigoplus_{i\in F} \bbGamma_i\bigr)\times \bigl(\bigoplus_{i\in I\setminus F} \bbGamma_i\bigr)$ which holds for all finite $\emptyset\neq F\subseteq I$. Upper bounds $\le $ in \eqref{eq9} can be directly showed as in the first paragraph of the proof of Proposition \ref{prop2}.
\end{proof}

We end with an example, which shows that knowing the exact value of Cowling-Haagerup constant (not just an upper and lower bound), can make a significant difference.

\begin{example}\label{ex1}
Let $(\bbGamma_n)_{n\in\NN}$ be a sequence of discrete quantum groups, such that $\Lambda_{cb}(\bbGamma_n)<+\infty$ for all $n\in\NN$ and $\liminf_{n\in\NN}\Lambda_{cb}(\bbGamma_n)>1$. Define $\bbGamma=\bigoplus_{n=1}^{\infty} \bbGamma_n$. Then using Corollary \ref{cor1} we calculate $\Lambda_{cb}(\bbGamma)=\prod_{n=1}^{\infty} \Lambda_{cb}(\bbGamma_n)=\infty$, hence $\bbGamma$ is not weakly amenable. Note that we wouldn't be able to conclude this knowing only $\Lambda_{cb}(\bbGamma_n\times \bbGamma_m)\ge \max(\Lambda_{cb}(\bbGamma_n),\Lambda_{cb}(\bbGamma_m))$. Since weak amenability implies Haagerup-Kraus approximation property AP (\cite[Proposition 5.7]{DKV_ApproxLCQG}), all quantum groups $\bbGamma_n$ have AP and so does $\bbGamma$ (\cite[Proposition 7.5]{DKV_ApproxLCQG}).
\end{example}

\section{Acknowledgements}

I would like to express my gratitute to Matt Daws and Christian Voigt for discussing topics related to approximation properties of quantum groups. This work was partially supported by FWO grant 1246624N. 

\bibliographystyle{plain}
\bibliography{bibliography}

\begin{thebibliography}{10}

\bibitem{bmt}
E.~B{\'e}dos, G.~J. Murphy, and L.~Tuset.
\newblock Co-amenability of compact quantum groups.
\newblock {\em J. Geom. Phys.}, 40(2):130--153, 2001.

\bibitem{BlecherMerdy}
D.~P. Blecher and C.~Le~Merdy.
\newblock {\em Operator algebras and their modules---an operator space
  approach}, volume~30 of {\em London Mathematical Society Monographs. New
  Series}.
\newblock The Clarendon Press, Oxford University Press, Oxford, 2004.
\newblock Oxford Science Publications.

\bibitem{Brannan}
M.~Brannan.
\newblock Approximation properties for locally compact quantum groups.
\newblock In {\em Topological quantum groups}, volume 111 of {\em Banach Center
  Publ.}, pages 185--232. Polish Acad. Sci. Inst. Math., Warsaw, 2017.

\bibitem{BrownOzawa}
N.~P. Brown and N.~Ozawa.
\newblock {\em $C^*$-Algebras and Finite-Dimensional Approximations}.
\newblock Graduate Studies in Mathematics, Volume 88. American Mathematical
  Society, 2008.

\bibitem{Caspers}
M.~Caspers.
\newblock Weak amenability of locally compact quantum groups and approximation
  properties of extended quantum {$SU(1,1)$}.
\newblock {\em Comm. Math. Phys.}, 331(3):1041--1069, 2014.

\bibitem{CowlingHaagerup}
M.~Cowling and U.~Haagerup.
\newblock Completely bounded multipliers of the {F}ourier algebra of a simple
  {L}ie group of real rank one.
\newblock {\em Invent. Math.}, 96(3):507--549, 1989.

\bibitem{Crann}
J.~Crann.
\newblock Amenability and covariant injectivity of locally compact quantum
  groups {II}.
\newblock {\em Canad. J. Math.}, 69(5):1064--1086, 2017.

\bibitem{CrannInner}
J.~Crann.
\newblock Inner amenability and approximation properties of locally compact
  quantum groups.
\newblock {\em Indiana Univ. Math. J.}, 68(6):1721--1766, 2019.

\bibitem{Averaging}
M.~{Daws}, J.~{Krajczok}, and C.~{Voigt}.
\newblock {Averaging multipliers on locally compact quantum groups}.
\newblock {\em arXiv e-prints}, page arXiv:2312.13626, December 2023.

\bibitem{DKV_ApproxLCQG}
M.~Daws, J.~Krajczok, and C.~Voigt.
\newblock The approximation property for locally compact quantum groups.
\newblock {\em Adv. Math.}, 438:Paper No. 109452, 2024.

\bibitem{CCAP}
K.~De~Commer, A.~Freslon, and M.~Yamashita.
\newblock C{CAP} for universal discrete quantum groups.
\newblock {\em Comm. Math. Phys.}, 331(2):677--701, 2014.
\newblock With an appendix by Stefaan Vaes.

\bibitem{EffrosRuan}
E.~G. Effros and Z.-J. Ruan.
\newblock {\em Operator spaces}, volume~23 of {\em London Mathematical Society
  Monographs. New Series}.
\newblock The Clarendon Press, Oxford University Press, New York, 2000.

\bibitem{FreslonFreeProducts}
A.~Freslon.
\newblock A note on weak amenability for free products of discrete quantum
  groups.
\newblock {\em C. R. Math. Acad. Sci. Paris}, 350(7-8):403--406, 2012.

\bibitem{Freslon}
A.~Freslon.
\newblock Examples of weakly amenable discrete quantum groups.
\newblock {\em J. Funct. Anal.}, 265(9):2164--2187, 2013.

\bibitem{FreslonPermanence}
A.~Freslon.
\newblock Permanence of approximation properties for discrete quantum groups.
\newblock {\em Ann. Inst. Fourier (Grenoble)}, 65(4):1437--1467, 2015.

\bibitem{Haagerup}
U.~Haagerup.
\newblock Group {$C^*$}-algebras without the completely bounded approximation
  property.
\newblock {\em J. Lie Theory}, 26(3):861--887, 2016.

\bibitem{JungeNeufangRuan}
M.~Junge, M.~Neufang, and Z.-J. Ruan.
\newblock A representation theorem for locally compact quantum groups.
\newblock {\em Internat. J. Math.}, 20(3):377--400, 2009.

\bibitem{KrajczokPhD}
J.~Krajczok.
\newblock {\em Modular properties of locally compact quantum groups}.
\newblock PhD thesis, IMPAN, 2022.

\bibitem{KrausRuan}
J.~Kraus and Z.-J. Ruan.
\newblock Approximation properties for {K}ac algebras.
\newblock {\em Indiana Univ. Math. J.}, 48(2):469--535, 1999.

\bibitem{KustermansVaesVN}
J.~Kustermans and S.~Vaes.
\newblock Locally compact quantum groups in the von {N}eumann algebraic
  setting.
\newblock {\em Math. Scand.}, 92(1):68--92, 2003.

\bibitem{NeshveyevTuset}
S.~Neshveyev and L.~Tuset.
\newblock {\em Compact quantum groups and their representation categories},
  volume~20 of {\em Cours Sp\'{e}cialis\'{e}s [Specialized Courses]}.
\newblock Soci\'{e}t\'{e} Math\'{e}matique de France, Paris, 2013.

\bibitem{PodlesWoronowicz}
P.~Podle\'{s} and S.~L. Woronowicz.
\newblock Quantum deformation of {L}orentz group.
\newblock {\em Comm. Math. Phys.}, 130(2):381--431, 1990.

\bibitem{SoltanViselter}
P.~M. So{\l}tan and A.~Viselter.
\newblock A note on amenability of locally compact quantum groups.
\newblock {\em Canad. Math. Bull.}, 57(2):424--430, 2014.

\bibitem{Tomatsu}
R.~Tomatsu.
\newblock Amenable discrete quantum groups.
\newblock {\em J. Math. Soc. Japan}, 58(4):949--964, 2006.

\bibitem{VanDaeleHaar}
A.~Van~Daele.
\newblock The {H}aar measure on a compact quantum group.
\newblock {\em Proc. Amer. Math. Soc.}, 123(10):3125--3128, 1995.

\bibitem{Wang}
S.~Wang.
\newblock Tensor products and crossed products of compact quantum groups.
\newblock {\em Proc. London Math. Soc. (3)}, 71(3):695--720, 1995.

\bibitem{cqg}
S.~L. Woronowicz.
\newblock Compact quantum groups.
\newblock In {\em Sym\'etries quantiques ({L}es {H}ouches, 1995)}, pages
  845--884. North-Holland, Amsterdam, 1998.

\end{thebibliography}

\end{document}